\documentclass[11pt]{amsart}
\usepackage{verbatim, graphicx, rotating}
\usepackage{mdwlist}
\usepackage{enumerate}
\usepackage{hyperref}
\usepackage{url}
\usepackage{amsthm}

\usepackage {amssymb}

\input xy
\xyoption{all}
\input epsf

\newlength{\tabwidth}
\newlength{\tabheight}
\newlength{\tabrule}
\newlength{\tabwidthx}
\newlength{\tabheightx}

\def\gentabbox#1#2#3#4{\vbox to \tabheight{\setlength{\tabrule}{#3}%
  \setlength{\tabwidthx}{#1\tabwidth}\addtolength{\tabwidthx}{\tabrule}%

\setlength{\tabheightx}{#2\tabheight}\addtolength{\tabheightx}{-\tabheight}%
  \hbox to #1\tabwidth{%
    \hspace{-0.5\tabrule}\rule{\tabrule}{#2\tabheight}\hspace{-\tabrule}%
    \vbox to #2\tabheight{\hsize=\tabwidthx%
      \vspace{-0.5\tabrule}\hrule width\tabwidthx height\tabrule%
      \vspace{-0.5\tabrule}\vfil%
      \hbox to \tabwidthx{\hss#4\hss}%
        \vfil\vspace{-0.5\tabrule}%
      \hrule width\tabwidthx height\tabrule\vspace{-0.5\tabrule}}%
    \hspace{-\tabrule}\rule{\tabrule}{#2\tabheight}\hspace{-0.5\tabrule}}%
  \vspace{-\tabheightx}}}
\def\genblankbox#1#2{\vbox to \tabheight{\vfil\hbox to
#1\tabwidth{\hfil}}}
\def\tabbox#1#2#3{\gentabbox{#1}{#2}{0.4pt}{\strut #3}}

\catcode`\:=13 \catcode`\.=13 \catcode`\;=13 
\catcode`\>=13 \catcode`\^=13
\def:#1\\{\hbox{$#1$}}
\def.#1{\tabbox{1}{1}{$#1$}}
\def>#1{\tabbox{2}{1}{$#1$}}
\def^#1{\tabbox{1}{2}{$#1$}}
\def;{\genblankbox{1}{1}\relax}
\catcode`\:=12 \catcode`\.=12 \catcode`\;=12 
\catcode`\>=12 \catcode`\^=7

\newcommand{\field}{\mathbb}
\newcommand{\liealgebra}{\mathfrak}
\newcommand{\la}{\liealgebra}

\newcommand{\C}{{\field C}}

\newcommand{\Z}{{\field Z}}

\renewcommand{\b}{\liealgebra b}

\newcommand{\n}{{\la n}}

\newcommand{\ga}{\alpha}

\newtheorem{prop}{Proposition}[section]

\newtheorem{theorem}[prop]{Theorem}

\newtheorem{corollary}[prop]{Corollary}
\newtheorem{proposition}[prop]{Proposition}

\theoremstyle{definition}

\newtheorem{remark}[prop]{Remark}
\newtheorem{example}[prop]{Example}

\newtheorem*{question}{Question}

\newtheorem{definition}[prop]{Definition}

\newcommand{\frg}{\mathfrak{g}}

\newcommand{\frt}{\mathfrak{t}}

\newcommand{\caX}{\mathcal{X}}

\begin{document}
\title[$K$-orbits on $G/B$ and Schubert structure constants]
{$K$-orbits on $G/B$ and Schubert constants for pairs of signed shuffles in types $C$ and $D$}

\author{Benjamin J. Wyser}
\date{\today}


\begin{abstract}
We give positive descriptions for certain Schubert structure constants $c_{u,v}^w$ for the full flag variety in Lie types $C$ and $D$.  This is accomplished by first observing that a number of the $K=GL(n,\C)$-orbit closures on these flag varieties coincide with Richardson varieties, and then applying a theorem of M. Brion on the decomposition of such an orbit closure in the Schubert basis in terms of paths in the weak order graph.
\end{abstract}

\maketitle

\section{Introduction}
Let $G$ be a simple algebraic group over $\C$, of classical type, with $B,B^- \subseteq G$ opposed Borel subgroups.  Let $W$ be the Weyl group for $G$.  For each $w \in W$, there exists a Schubert class $S_w = [\overline{B^-wB/B}] \in H^*(G/B)$.  It is well-known that the classes $\{S_w\}_{w \in W}$ form a $\Z$-basis for $H^*(G/B)$.  As such, for any $u,v \in W$, we have
\[ S_u \cdot S_v = \displaystyle\sum_{w \in W} c_{u,v}^w S_w \]
in $H^*(G/B)$, for uniquely determined non-negative integers $c_{u,v}^w$.  These integers are the \textit{Schubert structure constants}.

Although the Schubert constants are computable, it has been a long-standing open problem, even in type $A$, to give a positive (i.e. subtraction-free) formula for an arbitrary constant $c_{u,v}^w$ in terms of $u,v,w$.  Such positive formulas are known in type $A$ in various special cases, but fewer results of this sort are known in the other classical types.

In \cite{Wyser-11a}, a special case rule for structure constants $c_{u,v}^w$ in type $A_{n-1}$ is described in the event that $(u,v)$ form what is referred to there as a ``$(p,q)$-pair" ($p+q=n$).  The key observation of that paper is that a number of the $K=GL(p,\C) \times GL(q,\C)$-orbit closures on the flag variety coincide with \textit{Richardson varieties} --- intersections of Schubert varieties with opposite Schubert varieties.  The rule follows when this observation is combined with a theorem of M. Brion (Theorem \ref{thm:brion}), which describes intersection numbers of spherical subgroup orbit closures with Schubert varieties in terms of paths in the weak order graph.  Theorem \ref{thm:brion} applies to $K$-orbit closures since $K$ is a \textit{symmetric} subgroup of $GL(n,\C)$ (i.e. the fixed-point subgroup of an involution), and because symmetric subgroups form a special class of spherical subgroups.

In this note, we extend the results of \cite{Wyser-11a} to obtain analogous special case rules for some $c_{u,v}^w$ in Lie types $C$ and $D$.   If elements of $W$ are viewed as signed permutations, then the types of structure constants described by these rules correspond to pairs $(u,v)$ of ``signed shuffles".  The key observation once again is that a number of symmetric subgroup orbit closures on these flag varieties coincide with Richardson varieties.  More specifically, let $G = Sp(2n,\C)$ or $SO(2n,\C)$, and let $X=G/B$ be the flag variety for $G$.  Let $G' = GL(2n,\C)$, and let $X'=G'/B'$ be the flag variety for $G'$.  When $K' = GL(n,\C) \times GL(n,\C) \subseteq G'$ is intersected with $G$, the result is a symmetric subgroup of $G$, isomorphic to $K=GL(n,\C)$ in each case.  Moreover, the intersection of any $K'$-orbit on $X'$ with $X$, if non-empty, is a single $K$-orbit on $X$.  Using this, along with the fact that a number of the $K'$-orbit closures on $X'$ coincide with Richardson varieties in $X'$, we see also that a number of the $K$-orbit closures on $X$ coincide with Richardson varieties in $X$.  Specifically, if $Y'=\overline{Q'}$ is a $K'$-orbit closure on $X'$ coinciding with a Richardson variety in $X'$, and if $Q' \cap X \neq \emptyset$, then $Y:=Y' \cap X$ is a $K$-orbit closure on $X$ which coincides with a Richardson variety in $X$.  Because Theorem \ref{thm:brion} applies to the class of \textit{any} spherical subgroup orbit closure in \textit{any} flag variety, we apply it once again in these settings to obtain the additional rules.

The paper is organized as follows.  Section 1 is the introduction.  In Section 2, we cover some preliminaries, recalling Theorem \ref{thm:brion} and the results of \cite{Wyser-11a} which will be relevant to us here.  With these facts recalled, we describe the results in types $C$ and $D$ in sections 3 and 4, respectively.  We conclude with a natural question in Section 5:  Are there other spherical subgroups of the classical groups some of whose orbit closures happen to coincide with Richardson varieties?

The results presented here and in \cite{Wyser-11a} grew out of the author's doctoral thesis work on some aspects of the equivariant geometry of symmetric subgroup orbit closures on flag varieties.  The author thanks William A. Graham, his research advisor at the University of Georgia, for his invaluable assistance in that project, as well as for his help in editing and revising this manuscript.  The author also thanks an anonymous referee for his/her careful reading and helpful suggestions.

\section{Preliminaries}\label{sect:prelims}
\subsection{Schubert varieties, opposite Schubert varieties, and Richardson varieties}
Let $G,B,B^-,$ and $W$ be as in the introduction.  We quickly define our conventions regarding Schubert varieties, opposite Schubert varieties, and Schubert classes.
\begin{definition}
For $w \in W$, the \textbf{Schubert variety} $X_w$ is defined to be $\overline{BwB/B}$.  This is an irreducible subvariety of $G/B$ of complex dimension $l(w)$.  The \textbf{opposite Schubert variety} $X^w$ is defined to be $\overline{B^-wB/B}$.  This is an irreducible subvariety of $G/B$ of complex \textit{codimension} $l(w)$.  The \textbf{Schubert class} $S_w$ is defined to be the (Poincar\'e dual to the) fundamental class of $X^w$, i.e. $S_w = [X^w]$.  Note that $S_w \in H^{2l(w)}(G/B)$.
\end{definition}

\begin{definition}
For $u,v \in W$, the \textbf{Richardson variety} $X_u^v$ is defined to be $X_u \cap X^v$.  This intersection is non-empty if and only if $u \geq v$ in the Bruhat order on $W$.  In that event, the intersection is proper and reduced, and has dimension $l(u) - l(v)$.
\end{definition}

Due to the fact that $[X_w] = [X^{w_0w}]$ for any $w \in W$ ($w_0$ denoting the longest element of $W$), along with the fact that $X_u^v$ is a proper, reduced intersection, we have the following in $H^*(G/B)$:

\begin{equation}
	[X_u^v] = [X^{w_0u}] \cdot [X^v] = S_{w_0u} \cdot S_v.
\end{equation}

\subsection{$K$-orbits on $G/B$, the weak order, and a theorem of Brion}\label{sect:weak_order}
Let $G$ be any complex, reductive algebraic group, with $\theta: G \rightarrow G$ an involution, and $K=G^{\theta}$ the corresponding symmetric subgroup.  Let $T \subseteq B$ be a $\theta$-stable maximal torus and Borel subgroup, respectively.  $K$ acts on $G/B$ with finitely many orbits (\cite{Matsuki-79}).  Let $Q$ be one of these orbits.  Let $\ga \in \Delta(G,T)$ be a simple root, with $P_{\ga}$ the standard minimal parabolic subgroup of type $\ga$, and
\[ \pi_{\ga}: G/B \rightarrow G/P_{\ga} \]
the natural projection.  The set $\pi_{\ga}^{-1}(\pi_{\ga}(Q))$ is $K$-stable, and contains a dense $K$-orbit.  We denote the dense orbit as follows:
\[ s_{\ga} \cdot Q = \text{ the (unique) dense $K$-orbit on $\pi_{\ga}^{-1}(\pi_{\ga}(Q))$}. \]

If $\dim(\pi_{\ga}(Q)) < \dim(Q)$, then $s_{\ga} \cdot Q = Q$, but if $\dim(\pi_{\ga}(Q)) = \dim(Q)$, $s_{\ga} \cdot Q$ is another orbit of dimension one higher than the dimension of $Q$.

\begin{definition}
The \textbf{weak closure order} (or simply the \textbf{weak order}) on $K$-orbits is the partial order generated by relations of the form $Q \prec Q' \Leftrightarrow Q' = s_{\ga} \cdot Q \neq Q$, for $\ga \in \Delta$.  Equivalently, we may speak of the weak ordering on orbit \textit{closures}.  Supposing that $Y,Y'$ are the closures of orbits $Q,Q'$, respectively, we say that $Y' = s_{\ga} \cdot Y$ if and only if $Q' = s_{\ga} \cdot Q$, if and only if $Y' = \pi_{\ga}^{-1}(\pi_{\ga}(Y))$.
\end{definition} 

Let $Q,Q',Y,Y'$ be as above.  If $Y' = s_{\ga} \cdot Y \neq Y$, then the simple root $\ga$ can be categorized as either ``complex" or ``non-compact imaginary" for $Y$.  (See \cite{Richardson-Springer-90} for this terminology.)  In the non-compact imaginary case, $\ga$ is said to be of ``type I" or ``type II" depending on whether $Q$ is fixed by the ``cross action" of $s_{\ga}$.

\begin{definition}
The \textbf{cross action}, denoted $\times$, of $W$ on $K \backslash G/B$ is defined as follows:
\[ w \times (K \cdot gB) = K \cdot g\dot{w}^{-1} B, \]
where $\dot{w}$ denotes a representative of $w$ in $N_G(T)$.
 
If $s_{\ga} \cdot Y \neq Y$, and if $\ga$ is non-compact imaginary for $Y$, then $\ga$ is said to be of \textbf{type I} if $s_{\ga} \times Q \neq Q$, and of \textbf{type II} if $s_{\ga} \times Q = Q$.
\end{definition}

\begin{remark}
Note that the cross action is independent of the choice of $\dot{w}$, since for $t \in T$,
\[ (\dot{w}t)^{-1}B = t^{-1} \dot{w}^{-1}B = \dot{w}^{-1} (\dot{w} t^{-1} \dot{w}^{-1})B = \dot{w}^{-1} t' B = \dot{w}^{-1} B, \]
since $t' = \dot{w} t^{-1} \dot{w}^{-1}$ is an element of $T$ (hence an element of $B$), $\dot{w}$ being an element of $N_G(T)$.
\end{remark}

In \cite{Brion-01}, the poset graph for the set of orbit closures equipped with the weak order is endowed with additonal data, as follows:  Whenever $Y' = s_{\ga} \cdot Y$, the directed edge originating at $Y$ and terminating at $Y'$ gets a label of $\ga$.  Moreover, if $\ga$ is non-compact imaginary type II for $Q$, this edge is double.  (In all other cases, the edge is simple.)  The double edge is meant to indicate that when $\ga$ is a non-compact imaginary type II root, the restriction $\pi_{\ga}|_Y$ has degree $2$ over its image.  (In all other cases, $\pi_{\ga}|_Y$ is birational over its image.)

If $w \in W$, with $s_{i_1} \hdots s_{i_k}$ a reduced expression for $w$, set 
\[ w \cdot Y = s_{i_1} \cdot (s_{i_2} \cdot \hdots (s_{i_k} \cdot Y) \hdots ). \]
This is well-defined, independent of the choice of reduced expression for $w$, and defines an action of a certain monoid $M(W)$ on the set of $K$-orbit closures (\cite{Richardson-Springer-90}).  As a set, the monoid $M(W)$ is comprised of elements $m(w)$, one for each $w \in W$.  The multiplication on $M(W)$ is defined inductively by
\[ m(s)m(w) = 
\begin{cases}
	m(sw) & \text{ if $l(sw) > l(w)$}, \\
	m(w) & \text{ otherwise.}
\end{cases} \]

(For the sake of simplicity, we denote this action by $w \cdot Y$, as opposed to $m(w) \cdot Y$.  However, we emphasize that this defines an action of $M(W)$, and not of $W$.)

Suppose that $Y$ is a $K$-orbit closure on $G/B$ of codimension $d$.  Define the following subset of $W$:
\[ W(Y) := \{ w \in W \mid w \cdot Y = G/B \text{ and } l(w) = d\}. \]
(Note that in this definition, ``$G/B$" refers to the closure of the dense, open orbit.)  Elements of $W(Y)$ are precisely those $w$ such that there is a path connecting $Y$ to the top vertex of the weak order graph, the product of whose edge labels is $w$.  For any $w \in W(Y)$, denote by $D(w)$ the number of double edges in such a path.  (Although there may be more than one such path, each corresponding to a different reduced expression for $w$, any such path has the same number of double edges, so $D(w)$ \textit{is} well-defined.  See \cite[Lemma 5]{Brion-01}.)

With all of this defined, we now recall a theorem of \cite{Brion-01} which is used in \cite{Wyser-11a} to obtain positive rules for certain Schubert constants $c_{u,v}^w$ in type $A$.  In the present paper, we use it again to obtain analogous rules in types $CD$.
\begin{theorem}[\cite{Brion-01}]\label{thm:brion}
Let $Y$ be a $K$-orbit closure on $G/B$.  In $H^*(G/B)$, the fundamental class of $Y$ is expressed in the Schubert basis as follows:
\[ [Y] = \displaystyle\sum_{w \in W(Y)} 2^{D(w)} S_w. \]
\end{theorem}

\subsection{$GL(p,\C) \times GL(q,\C)$-orbits, $(p,q)$-clans, and $(p,q)$-pairs}
We now briefly recall results of \cite{Wyser-11a} which relate $GL(p,\C) \times GL(q,\C)$-orbit closures on the type $A$ flag variety to Richardson varieties.

Let $G = GL(n,\C)$, with $B$ the Borel subgroup of $G$ consisting of upper-triangular matrices.  Let $X = G/B$ be the type $A$ flag variety.  For $p,q$ with $p+q=n$, let $\theta = \text{int}(I_{p,q})$, where $I_{p,q}$ is the matrix
\[ \begin{pmatrix}
I_p & 0 \\
0 & -I_q \end{pmatrix}, \]
and where $\text{int}(g)$ denotes conjugation by $g$.

One checks easily that $K = G^{\theta}$ is isomorphic to $GL(p,\C) \times GL(q,\C)$, embedded diagonally as follows:
\[ K = \left\{
\left[
\begin{array}{cc}
K_{11} & 0 \\
0 & K_{22} \end{array}
\right] \in G
\ \middle\vert \ 
\begin{array}{c}
K_{11} \in GL(p,\C) \\
K_{22} \in GL(q,\C) \end{array}
\right\} .\]

As detailed in \cite{Matsuki-Oshima-90,Yamamoto-97,McGovern-Trapa-09}, and as recalled in \cite{Wyser-11a}, the set $K \backslash X$ of $K$-orbits is in bijection with the set of combinatorial objects known as $(p,q)$-clans:
\begin{definition}\label{def:clans}
A \textbf{$(p,q)$-clan} is a string of $n=p+q$ symbols, each of which is a $+$, a $-$, or a natural number.  The string must satisfy the following two properties:
\begin{enumerate}
	\item Every natural number which appears must appear exactly twice in the string.
	\item The difference in the number of plus signs and the number of minus signs in the string must be $p-q$.  (If $q > p$, then there should be $q-p$ more minus signs than plus signs.)
\end{enumerate}
\end{definition}

Such strings are considered only up to an equivalence which says, essentially, that it is the positions of matching natural numbers, rather than the actual values of the numbers, which determine a clan.  So, for instance, the clans $(1,2,1,2)$, $(2,1,2,1)$, and $(5,7,5,7)$ are all the same, since they all have matching natural numbers in positions $1$ and $3$, and also in positions $2$ and $4$.  On the other hand, $(1,2,2,1)$ is a different clan, since it has matching natural numbers in positions $1$ and $4$, and in positions $2$ and $3$.

For $(p,q)$-clans, there is an obvious notion of pattern avoidance.  This notion is used in, e.g., \cite{McGovern-09a,McGovern-Trapa-09} to give combinatorial criteria for rational smoothness of symmetric subgroup orbit closures in various cases.  In \cite{Wyser-11a}, pattern avoidance is used to identify certain $K$-orbit closures as Richardson varieties.

\begin{definition}
Given a $(p,q)$-clan $\gamma$ and a $(p',q')$-clan $\gamma'$ (with $p' \leq p$ and $q' \leq q$), $\gamma$ is said to \textbf{avoid the pattern} $\gamma'$ if there is no substring of $\gamma$ of length $p'+q'$ which is equal to $\gamma'$ as a clan.
\end{definition}

One of the main observations of \cite{Wyser-11a} is that the closure of any $K'$-orbit whose clan avoids the pattern $(1,2,1,2)$ is a Richardson variety.  The precise statement is as follows:
\begin{theorem}[{\cite[Theorem 6.4]{Wyser-11a}}]\label{thm:type-a-richardson}
Suppose that $\gamma=(c_1,\hdots,c_n)$ is a $(p,q)$-clan avoiding $(1,2,1,2)$.  Define permutations $u(\gamma),v(\gamma)$ as follows:
\begin{itemize}
	\item $u(\gamma)$ is the permutation which assigns the numbers $p,p-1,\hdots,1$, in order, to those positions $i$ for which $c_i$ is either a $+$ or the second occurrence of some natural number, and the numbers $n,n-1,\hdots,p+1$, in order, to the remaining positions.
	\item $v(\gamma)$ is the permutation which assigns the numbers $1,\hdots,p$, in order, to those positions $i$ for which $c_i$ is either a $+$ or the first occurrence of some natural number, and the numbers $p+1,\hdots,n$, in order, to the remaining positions.
\end{itemize}
Then the orbit closure $Y_{\gamma} = \overline{Q_{\gamma}}$ is the Richardson variety $X_{u(\gamma)}^{v(\gamma)}$.
\end{theorem}

As an example, if $\gamma = (+,-,1,1,2,2)$, then $Y_{\gamma} = X_{365241}^{142536}$.

Let $u=u(\gamma)$, $v=v(\gamma)$.  Since $[Y_{\gamma}]$ is the product $[X_{w_0u}] \cdot [X_v]$, Theorem \ref{thm:brion} and some case-specific knowledge of the combinatorics of $K \backslash X$ give a positive rule for computing $c_{w_0u,v}^w$ for any $w$.  Note that $(w_0u,v)$ have the property that the one-line notation for $w_0u$ is a ``shuffle" of $1,\hdots,q$ and $q+1,\hdots,n$ --- that is, in the one-line notation, $1,\hdots,q$ occur in order, and $q+1,\hdots,n$ occur in order.  Likewise, $v$ is a shuffle of $1,\hdots,p$ and $p+1,\hdots,n$.  In \cite{Wyser-11a}, $(w_0u,v)$ is referred to as a ``$(p,q)$-pair".

There is a converse to Theorem \ref{thm:type-a-richardson}, which says that any Richardson variety $X_u^v$ such that $(w_0u,v)$ form a $(p,q)$-pair is the closure of a $K$-orbit on $X$.

\begin{prop}[{\cite[Proposition 7.2]{Wyser-11a}}]\label{prop:type-a-richardson-converse}
Suppose that $(w_0u,v)$ is a $(p,q)$-pair with $u \geq v$.  The Richardson variety $X_u^v$ is the $K$-orbit closure $Y_{\gamma(u,v)}$, where $\gamma(u,v)$ is a $(p,q)$-clan avoiding the pattern $(1,2,1,2)$.  The clan $\gamma(u,v)$ is produced from $u,v$ by the following recipe:  First, create an ``FS-pattern" $(e_1,\hdots,e_n)$ from $u,v$ as follows:
\begin{enumerate}
	\item If $u(i),v(i) \leq p$, set $e_i = +$.
	\item If $u(i),v(i) > p$, set $e_i = -$.
	\item If $u(i) > p$, $v(i) \leq p$, set $e_i = F$.
	\item If $u(i) \leq p$, $v(i) > p$, set $e_i = S$.
\end{enumerate}

From this FS-pattern, produce $\gamma(u,v)=(c_1,\hdots,c_n)$ by the following steps:
\begin{enumerate}
	\item First, for all $i$ with $e_i = \pm$, set $c_i = e_i$.
	\item Next, for all $i$ with $e_i = F$, set $c_i$ to be a distinct natural number.  (If there are $m$ occurrences of F, these may as well be the numbers $1,\hdots,m$, in order from left to right.)
	\item Finally, starting at the left and moving to the right, for all $i$ with $e_i = S$, set $c_i$ to be the mate of the closest natural number to the left of position $i$ which does not yet have a mate.
\end{enumerate}
\end{prop}

As an example, consider the Richardson variety $X_{365241}^{142536}$.  The FS-pattern produced is $(+,-,F,S,F,S)$.  The clan produced from this FS-pattern by the steps outlined above is:
\begin{enumerate}
	\item $(+,-,*,*,*,*)$
	\item $(+,-,1,*,2,*)$
	\item $(+,-,1,1,2,2)$
\end{enumerate}

Thus we recover the fact, noted above, that $X_{365241}^{142536} = Y_{(+,-,1,1,2,2)}$.

Proposition \ref{prop:type-a-richardson-converse} tells us that the rule of \cite[Theorem 7.5]{Wyser-11a} for structure constants in type $A$ applies to all $c_{u,v}^w$ where $u,v$ form a $(p,q)$-pair with $w_0u \geq v$.  (Note that if $w_0u \not\geq v$, then $c_{u,v}^w$ is automatically zero.)

\section{Type $C$}\label{sect:type-c}
In this section, we apply the results recalled in the previous section to obtain a positive rule for Schubert constants in type $C$.  In the next section, we do the same in type $D$.

We realize the complex symplectic group $G = Sp(2n,\C)$ as the isometry group of the symplectic form
\[ \left\langle x,y \right\rangle = \displaystyle\sum_{i=1}^n x_i y_{2n+1-i} - \displaystyle\sum_{i=n+1}^{2n} x_i y_{2n+1-i}. \]
That is, $G$ is the set of all matrices $g$ such that $g^t J_{n,n} g = J_{n,n}$, where
\[ J_{n,n} = 
\begin{pmatrix}
0 & J_n \\
-J_n & 0
\end{pmatrix}, \]
with $J_n$ the $n \times n$ antidiagonal matrix $(\delta_{i,n+1-j})$.

Let $B \subseteq G$ be the Borel subgroup of upper triangular elements of $G$.  The flag variety $G/B$ naturally identifies with the set of flags which are isotropic with respect to the form $\langle \cdot, \cdot \rangle$.  (A flag $F_{\bullet} = (F_0 \subset F_1 \subset \hdots \subset F_{2n})$ is isotropic with respect to $\langle \cdot, \cdot \rangle$ if and only if $F_{2n+1-i} = F_i^{\perp}$ for $i=1,\hdots,n$, where orthogonal complements are taken with respect to $\langle \cdot, \cdot \rangle$.)  Thus the type $C$ flag variety $X$ is naturally a closed subvariety of the type $A$ flag variety $X'=G'/B'$, where $G' = GL(2n,\C)$ and $B' \subseteq G'$ is the Borel subgroup of upper triangular elements of $G'$.  (The embedding $X' \hookrightarrow X$ corresponds to the obvious map $G/B \hookrightarrow G'/B'$ taking $gB$ to $gB'$.)

Recall that $K' = GL(n,\C) \times GL(n,\C) \subseteq G'$ is $(G')^{\theta'}$ for $\theta' = \text{int}(I_{n,n})$.  One checks easily that $G$ is stable under $\theta'$, so that $\theta := \theta'|_{G}$ is an involution of $G$.  Let $K = G^{\theta} = G \cap K'$.  It is a straightforward calculation to see that
\[ K = 
\left\{ 
\begin{pmatrix}
g & 0 \\
0 & J_n \ (g^t)^{-1} J_n
\end{pmatrix}
\ \middle\vert \ 
g \in GL(n,\C) \right\} \cong GL(n,\C). \]

We wish to see that certain of the $K$-orbit closures on $X$ are Richardson varieties.  Namely, suppose that $Y'=\overline{Q'}$ is a $K'$-orbit closure on $X'$ which is a Richardson variety, and suppose that $Q' \cap X = Q \neq \emptyset$.  Then $Y=\overline{Q}$ is a $K$-orbit closure on $X$ which coincides with a Richardson variety.

To see this, we must first note that in the above notation, $Q$ is a $K$-orbit on $X$.  Thus we start by identifying the $K$-orbits on $X$.  It is clear that the intersection of a $K'$-orbit on $X'$ with $X$, if non-empty, is stable under $K$ and hence is at least a union of $K$-orbits on $X$.  In fact, each such non-empty intersection is a \textit{single} $K$-orbit on $X$.  (We briefly describe the idea of the proof of this below, see Proposition \ref{prop:skew-symmetric-intersections}).  This means that $K$-orbits on $X$ are in one-to-one correspondence with $K'$-orbits on $X'$ which intersect $X$.  And since $K'$-orbits on $X'$ are parametrized by $(n,n)$-clans, $K$-orbits on $X$ are parametrized by the subset of $(n,n)$-clans $\gamma$ having the property that $Q_{\gamma}' \cap X \neq \emptyset$.  ($Q_{\gamma}'$ denotes the $K'$-orbit on $X'$ corresponding to $\gamma$.)  As it turns out, this last condition amounts to $\gamma$ possessing a simple combinatorial property.

\begin{definition}
An $(n,n)$-clan $\gamma=(c_1,\hdots,c_{2n})$ is \textbf{skew-symmetric} if the clan $(c_{2n},\hdots,c_1)$ obtained from $\gamma$ by reversing the string is the ``negative" of $\gamma$, meaning it is the same clan except with all signs flipped.  More precisely, $\gamma$ is skew-symmetric if and only if for each $i=1,\hdots,n$,
\begin{enumerate}
	\item If $c_i$ is a sign, then $c_{2n+1-i}$ is the opposite sign.
	\item If $c_i$ is a number, then $c_{2n+1-i}$ is also a number, and if $c_{2n+1-i} = c_j$, then $c_i = c_{2n+1-j}$.
\end{enumerate}
\end{definition}

\begin{prop}\label{prop:skew-symmetric-intersections}
Let $\gamma$ be an $(n,n)$-clan, with $Q_{\gamma}'$ the corresponding $K'$-orbit on $X'$.  Then $Q_{\gamma}' \cap X \neq \emptyset$ if and only if $\gamma$ is skew-symmetric.  Furthermore, if $\gamma$ is skew-symmetric, then $Q_{\gamma}' \cap X$ is a \textit{single} $K$-orbit on $X$.  Thus $K \backslash X$ is parametrized by the set of all skew-symmetric $(n,n)$-clans.
\end{prop}
\begin{proof}
The first claim is proved in \cite{Yamamoto-97}.  We do not prove the second claim here, but briefly indicate what is involved in one possible proof.  (The author thanks Peter Trapa for explaining this general line of argument to him.)  One considers an entire ``inner class" of involutions $\theta_i$, and lets $K_i = G^{\theta_i}$.  (To be more specific, in this case, the relevant $K_i$'s turn out to be $GL(n,\C)$ together with the groups $Sp(2p,\C) \times Sp(2q,\C)$ as $p,q$ range over all possibilities with $p+q=n$.)  Each $K_i$ can be realized as $G \cap K_i'$, with each $K_i'$ isomorphic to $GL(p,\C) \times GL(q,\C)$ for some $p,q$.  Given this setup, one can consider the disjoint union of all orbit sets, $\coprod_i K_i \backslash X$.  The resulting set of orbits is in bijection with the so-called ``one-sided parameter space" $\caX$, defined in \cite{Adams-DuCloux-09}.  The cardinality of the latter set can be computed explicitly, allowing for a counting argument which shows that no intersection of a $K_i'$-orbit on $X'$ with $X$ can be anything other than a single $K_i$-orbit (for any $K_i$).  

A detailed argument of this type will appear in another paper currently in preparation by the author.
\end{proof}

Suppose that $Y_{\gamma}' = \overline{Q_{\gamma}'}$ is a $K'$-orbit closure on $X'$, with $\gamma$ a skew-symmetric $(n,n)$-clan avoiding the pattern $(1,2,1,2)$.  We know that $Y_{\gamma}'$ is a Richardson variety in $X'$.  By the previous proposition, we also know that $Y_{\gamma} = \overline{Q_{\gamma}} = \overline{Q_{\gamma}' \cap X} = Y_{\gamma}' \cap X$.  We now wish to see that $Y_{\gamma}$ is a Richardson variety in $X$.

Let $W$ be the Weyl group for $G$, and $W' = S_{2n}$ the Weyl group for $G'$.  Recall that $W$ consists of signed permutations of $[n]$ (changing any number of signs).  A signed permutation of $[n]$ is a bijection $\sigma$ from the set $\{\pm 1,\hdots,\pm n\}$ to itself having the property that
\[ \sigma(-i) = -\sigma(i) \]
for all $i$.  We shall denote signed permutations in one-line notation with bars over some of the numbers to indicate negative values.  For example, the one-line notation $3\overline{1}2$ represents the signed permutation which sends $1$ to $3$, $2$ to $-1$, and $3$ to $2$.

There is a natural embedding $W \hookrightarrow W'$ as permutations $\pi$ of $[2n]$ having the property that
\[ \pi(2n+1-i) = 2n+1-\pi(i) \]
for all $i \in [n]$.  This embedding takes a signed permutation $\sigma$ of $[n]$ to the permutation $\pi$ of $[2n]$ defined by
\[ \pi(i) = 
\begin{cases}
\sigma(i) & \text{ if $\sigma(i) > 0$} \\
2n+1-|\sigma(i)| & \text{ otherwise,}
\end{cases} \]
and $\pi(2n+1-i) = 2n+1-\pi(i)$ for $i=1,\hdots,n$.

To avoid any confusion in terminology, we will refer to elements of $S_{2n}$ in the image of this embedding as ``signed elements of $S_{2n}$".  For any $w \in W$, we will denote its image in $S_{2n}$ by $w'$.  (Conversely, any permutation denoted $w'$, $u'$, etc. should be assumed to be the image of a corresponding element $w$, $u$, etc. of $W$.)

\begin{prop}\label{prop:skew_symmetric_clans}
Suppose that $\gamma=(c_1,\hdots,c_{2n})$ is a skew-symmetric $(n,n)$-clan avoiding the pattern $(1,2,1,2)$.  Let $u(\gamma),v(\gamma) \in S_{2n}$ be as described in the statement of Theorem \ref{thm:type-a-richardson}.  Then $u(\gamma),v(\gamma)$ are signed elements of $S_{2n}$.
\end{prop}
\begin{proof}
Consider $u$.  Moving from left to right from $c_1$ to $c_n$, the $i$th character which is either a plus sign or a second occurrence of a natural number is assigned the value $n-i+1$.  Say this character is $c_k$, so that $u(k) = n-i+1$.  Then due to skew-symmetry, as we move right to left from $c_{2n}$ to $c_{n+1}$, $c_{2n+1-k}$ is the $i$th occurrence of a \textit{minus} sign or \textit{first} occurrence of a natural number, so it is assigned $n+i$, i.e. $u(2n+1-k) = n + i$.  This says that 
\[ u(2n+1-k) = n+i = 2n+1-(n-i+1) = 2n+1-u(k). \]

Likewise, moving left to right, the $i$th character which is either a minus sign or first occurrence of a natural number is assigned the value $2n-i+1$.  If this character is $c_k$, then $u(k) = 2n-i+1$.  Then moving right to left, the $i$th character which is either a plus sign or second occurrence of a natural number is $c_{2n+1-k}$, and this position is assigned the value $i$.  Thus $u(2n+1-k) = i$.  Then
\[ u(2n+1-k) = i = 2n+1-(2n-i+1) = 2n+1-u(k). \]

Thus $u(2n+1-k) = 2n+1-u(k)$ for all $k \in [n]$.  The argument for $v$ is virtually identical.
\end{proof}

The last proposition says in particular that any $K'$-orbit closure $Y'=\overline{Q'}$ which coincides with a Richardson variety, and which has the property that $Q' \cap X \neq \emptyset$, is of the form $X_{u'}^{v'}$ for $u,v \in W$.  Indeed, such an orbit closure is of the form $Y_{\gamma}'$ with $\gamma$ a skew-symmetric $(n,n)$-clan avoiding the pattern $(1,2,1,2)$.  We know that $Y_{\gamma}' = X_{u(\gamma)}^{v(\gamma)}$ by Theorem \ref{thm:type-a-richardson}, and we have just shown that $u(\gamma)$ and $v(\gamma)$ are of the form $u',v'$.  We wish to see now that the intersection of $Y_{\gamma}' = X_{u'}^{v'}$ with $X$ is the Richardson variety $X_u^v$.  We need the following fact regarding the Bruhat order on $W$, for which we refer to \cite[\S 3.3]{Billey-Lakshmibai-00}:

\begin{prop}[{\cite[\S 3.3]{Billey-Lakshmibai-00}}]
When $W$ is considered as a subset of $W'$ via the embedding described above, the Bruhat order on $W$ is the one induced by the Bruhat order on $W'$.
\end{prop}

\begin{corollary}\label{cor:type_c_schubert_vars}
Suppose $X_{w'}$ (resp. $X^{w'}$) is the Schubert subvariety (resp. the opposite Schubert subvariety) of $X'$ corresponding to $w'$, the image of $w \in W$ in $W'=S_{2n}$.  Then $X_{w'} \cap X = X_w$ (resp. $X^{w'} \cap X = X^w$), the Schubert subvariety (resp. the opposite Schubert subvariety) of $X$ corresponding to $w$.
\end{corollary}
\begin{proof}
For the sake of clarity, we denote Schubert cells of $X$ by $\{C_w\}_{w \in W}$, and Schubert cells of $X'$ by $\{D_w\}_{w \in W'}$.

One first checks that the $C_w$ are all of the form $D_{w'} \cap X$.  Indeed, each of these intersections is easily seen to be non-empty, and is clearly stable under the Borel $B = B' \cap G$, with $B' \subseteq GL(2n,\C)$ the Borel of upper-triangular elements of the larger group.  An easy counting argument then shows that each such intersection must be a single $B$-orbit, hence equal to $C_w$.

Then using the previous proposition, we see that
\[ X_{w'} \cap X = (\bigcup_{v \leq w'} D_v) \cap X = \]
\[ \bigcup_{v \leq w'} (D_v \cap X) = \]
\[ \bigcup_{v' \leq w', v \in W} (D_{v'} \cap X) = \]
\[ \bigcup_{v \leq w, v \in W} C_v = X_w, \]
and similarly for opposite Schubert varieties.
\end{proof}

\begin{corollary}\label{cor:type-c-richardson}
Let $\gamma$ be a skew-symmetric $(n,n)$-clan avoiding the pattern $(1,2,1,2)$.  Let $Y_{\gamma}'$ denote the closure of the $K'$-orbit $Q_{\gamma}'$ on $X'$ associated to $\gamma$, and let $Y_{\gamma}$ denote the closure of the $K$-orbit $Q_{\gamma} = Q_{\gamma}' \cap X$ on $X$ associated to $\gamma$.  Let $u'=u(\gamma)$, $v'=v(\gamma)$ be the elements of $S_{2n}$ produced from $\gamma$ as described in the statement of Theorem \ref{thm:type-a-richardson}.  Then $Y_{\gamma}$ is the Richardson variety $X_u^v$.
\end{corollary}
\begin{proof}
By Proposition \ref{prop:skew_symmetric_clans}, $u(\gamma),v(\gamma)$ are signed elements of $S_{2n}$, which justifies denoting them by $u',v'$.  We know that $Y_{\gamma}'$ is the Richardson variety $X_{u'}^{v'} = X_{u'} \cap X^{v'}$.  Since $Q_{\gamma} = Q_{\gamma}' \cap X$ by Proposition \ref{prop:skew-symmetric-intersections}, we have $Y_{\gamma} = Y_{\gamma}' \cap X$.  Then by Corollary \ref{cor:type_c_schubert_vars},
\[ Y_{\gamma} = Y_{\gamma}' \cap X = X_{u'} \cap X^{v'} \cap X = (X_{u'} \cap X) \cap (X^{v'} \cap X) = X_u \cap X^v =X_u^v. \]
\end{proof}

Recall that for an $(n,n)$-clan $\gamma$, the permutation $u(\gamma)$ is a shuffle of $n,n-1,\hdots,1$ and $2n,2n-1,\hdots,n+1$, while $v(\gamma)$ is a shuffle of $1,\hdots,n$ and $n+1,\hdots,2n$.  When $\gamma$ is skew-symmetric, then $u(\gamma)$ and $v(\gamma)$ are of the form $u',v'$ for $u,v \in W$ signed permutations.  Note that as signed permutations, $u$ is a shuffle of $n,n-1,\hdots,k+1$ and $-1,\hdots,-k$ for some $k$, while $v$ is a shuffle of $1,\hdots,j$ and $-n,\hdots,-(j+1)$ for some $j$.

Since $Y_{\gamma}$ is the Richardson variety $X_u^v$, using Theorem \ref{thm:brion}, we can compute the Schubert product $S_{w_0u} \cdot S_v$.  Given the above description of $u$, and given that $w_0$ is the signed permutation which flips all signs, $w_0u$ is a shuffle of $1,\hdots,k$ and $-n,\hdots,-(k+1)$.  Thus we make the following definition:

\begin{definition}
Suppose that $u,v \in W$ are signed permutations with the following properties:
\begin{enumerate}
	\item $u$ is a shuffle of $1,\hdots,k$ and $-n,\hdots,-(k+1)$ for some $k$.
	\item $v$ is a shuffle of $1,\hdots,j$ and $-n,\hdots,-(j+1)$ for some $j$.
\end{enumerate}
We call $(u,v)$ a \textbf{type $C$ pair of signed shuffles}.
\end{definition}

Since Theorem \ref{thm:brion} applies to the class of any spherical subgroup orbit closure in any flag variety, we now know that we can use it to compute the Schubert product $S_u \cdot S_v$ whenever $(u,v)$ is a type $C$ pair of signed shuffles such that $u = w_0 \cdot u(\gamma)$ and $v = v(\gamma)$ for some skew-symmetric $(n,n)$-clan $\gamma$.  Note that by Proposition \ref{prop:type-a-richardson-converse}, we can compute \textit{all} such non-trivial products this way:  Assuming $(u,v)$ is a type $C$ pair of signed shuffles with $w_0u \geq v$, we have
\[ S_u \cdot S_v = [X_{w_0u}] \cdot [X^v] = [X_{w_0u}^v] = [X_{(w_0u)'}^{v'} \cap X] = [X_{w_0'u'}^{v'} \cap X] = [Y_{\gamma(u',v')}' \cap X] = [Y_{\gamma(u',v')}]. \]

We now apply Theorem \ref{thm:brion} to determine a positive rule for type $C$ structure constants $c_{u,v}^w$ when $(u,v)$ is a type $C$ pair of signed shuffles.  To determine precisely what the rule says, we must understand the weak order on $K \backslash X$, and specifically the monoidal action of $M(W)$ on $K \backslash X$.  For this, we refer to \cite{Yamamoto-97,Matsuki-Oshima-90}.  Let $\frt$ be the Cartan subalgebra of $\frg = \text{Lie}(G)$ consisting of diagonal matrices 
\[ \text{diag}(a_1,\hdots,a_n,-a_n,\hdots,-a_1). \]
Let $x_1,\hdots,x_n$ be coordinates on $\frt$, with 
\[ x_i(\text{diag}(a_1,\hdots,a_n,-a_n,\hdots,-a_1)) = a_i. \]
Order the simple roots in the following way:  $\ga_i = x_i - x_{i+1}$ for $i=1,\hdots,n-1$, and $\ga_n = 2x_n$.  For $i=1,\hdots,n$, let $s_i$ denote $s_{\ga_i}$.  We wish to define the monoidal action of $s_i$ on any $K$-orbit $Q$.  Identifying $K$-orbits with skew-symmetric $(n,n)$-clans, we speak instead of the action of $s_i$ on such a clan $\gamma=(c_1,\hdots,c_{2n})$.  We identify the simple roots as complex or non-compact imaginary (type I or II) for $\gamma$, rather than for the orbit $Q_{\gamma}$.

First consider $\ga = \ga_i$ with $i=1,\hdots,n-1$.  Then $\ga$ is complex for $\gamma$ (and $s_i \cdot \gamma \neq \gamma)$ if and only if one of the following occurs:
\begin{enumerate}
	\item $c_i$ is a sign, $c_{i+1}$ is a number, and the mate for $c_{i+1}$ occurs to the right of $c_{i+1}$.
	\item $c_i$ is a number, $c_{i+1}$ is a sign, and the mate for $c_i$ occurs to the left of $c_i$.
	\item $c_i$ and $c_{i+1}$ are unequal natural numbers, the mate for $c_i$ occurs to the left of the mate for $c_{i+1}$, and $(c_i,c_{i+1}) \neq (c_{2n-i},c_{2n-i+1})$.
\end{enumerate}

In this event, $s_i \cdot \gamma = \gamma'$, where $\gamma'$ is obtained from $\gamma$ by interchanging the characters in positions $i,i+1$, \textit{and} the characters in positions $2n-i,2n-i+1$.

On the other hand, $\ga$ is non-compact imaginary for $\gamma$ if and only if one of the following two possibilities occurs:
\begin{enumerate}
	\item $c_i$ and $c_{i+1}$ are opposite signs.
	\item $c_i$ and $c_{i+1}$ are unequal natural numbers, with $(c_i,c_{i+1}) = (c_{2n-i},c_{2n-i+1})$ (in order).
\end{enumerate}

In the first case above, $s_i \cdot \gamma = \gamma''$, where $\gamma''$ is obtained from $\gamma$ by replacing the opposite signs in positions $i,i+1$ and (by skew-symmetry) positions $2n-i,2n-i+1$ each by a distinct pair of matching natural numbers.  In the second case, $s_i \cdot \gamma = \gamma'''$, where $\gamma'''$ is obtained from $\gamma$ by interchanging the characters in positions $i,i+1$ (but \textit{not} those in positions $2n-i,2n-i+1$).

The cross-action of $s_i$ on the clan $\gamma$ is by the action of the corresponding permutation $s_i' = (i,i+1)(2n-i,2n-i+1) \in S_{2n}$.  That is, $s_i \times \gamma$ is the clan obtained from $\gamma$ by interchanging the characters in positions $i,i+1$ and the characters in positions $2n-i,2n-i+1$.  Note that this action does not fix $\gamma$ in the first case above, since two pairs of opposite signs get interchanged.  This means that in the first case above, $\ga$ is a non-compact imaginary root of type I.  However, this cross-action \textit{does} fix $\gamma$ in the second case, since interchanging the two pairs of numbers does not change the clan.  (Recall the equivalence of character strings described immediately after Definition \ref{def:clans}.)  Thus in the second case, $\ga$ is a non-compact imaginary root of type II.

Now consider $s_n$.  We have that $\ga_n$ is complex for $\gamma$ (and $s_n \cdot \gamma \neq \gamma$) if and only if $c_n$ and $c_{n+1}$ are unequal natural numbers, with the mate for $c_n$ to the left of the mate for $c_{n+1}$.  $\ga_n$ is non-compact imaginary for $\gamma$ if and only if $c_n$ and $c_{n+1}$ are opposite signs.  In this case, the cross action of $s_n$ on $\gamma$ is by the permutation action of $s_n' = (n,n+1) \in S_{2n}$.  Thus $s_n \times \gamma \neq \gamma$, since the cross action of $s_n$ interchanges the opposite signs in positions $n,n+1$.  Thus we see that $\ga_n$ is a non-compact imaginary root of type I.

We summarize the preceding discussion by recasting the monoidal action of $W$ in purely combinatorial terms, as a sequence of ``operations" on skew-symmetric $(n,n)$-clans.  Given a simple reflection $s_i$, consider the following possible operations on $\gamma$:
\begin{enumerate}[(a)]
	\item Interchange the characters in positions $i,i+1$ \textit{and} the characters in positions $2n-i,2n-i+1$.
	\item Interchange the characters in positions $i,i+1$, \textit{but not} the characters in positions $2n-i,2n-i+1$.
	\item Replace the characters in positions $i,i+1$ \textit{and} the characters in positions $2n-i,2n-i+1$ each by a pair of matching natural numbers.
	\item Interchange the characters in positions $n,n+1$.
	\item Replace the characters in positions $n,n+1$ by a pair of matching natural numbers.
\end{enumerate}

Then the monoidal action of $M(W)$ on $\gamma$ is as follows:  For $i=1,\hdots,n-1$,
\begin{enumerate}
	\item If $c_i$ is a sign, $c_{i+1}$ is a number, and the mate for $c_{i+1}$ occurs to the right of $c_{i+1}$, $s_i \cdot \gamma$ is obtained from $\gamma$ by operation (a).
	\item If $c_i$ is a number, $c_{i+1}$ is a sign, and the mate for $c_i$ occurs to the left of $c_i$, $s_i \cdot \gamma$ is obtained from $\gamma$ by operation (a).
	\item If $c_i$ and $c_{i+1}$ are unequal natural numbers, with the mate for $c_i$ occurring to the left of the mate for $c_{i+1}$, \textit{and} if $(c_i,c_{i+1}) \neq (c_{2n-i},c_{2n-i+1})$, then $s_i \cdot \gamma$ is obtained from $\gamma$ by operation (a).
	\item If $c_i$ and $c_{i+1}$ are unequal natural numbers, and if $(c_i,c_{i+1}) = (c_{2n-i},c_{2n-i+1})$, then $s_i \cdot \gamma$ is obtained from $\gamma$ by operation (b). \textbf{(*)}
	\item If $c_i$ and $c_{i+1}$ are opposite signs, then $s_i \cdot \gamma$ is obtained from $\gamma$ by operation (c).
	\item If none of the above hold, them $s_i \cdot \gamma = \gamma$.
\end{enumerate}

We give examples of (1)-(5) above:
\begin{enumerate}
	\item $s_2 \cdot (+,+,1,1,2,2,-,-) = (+,1,+,1,2,-,2,-)$
	\item $s_2 \cdot (1,1,+,+,-,-,2,2) = (1,+,1,+,-,2,-,2)$
	\item $s_1 \cdot (1,2,1,2,3,4,3,4) = (2,1,1,2,3,4,4,3)$
	\item $s_1 \cdot (1,2,3,4,3,4,1,2) = (2,1,3,4,3,4,1,2)$
	\item $s_1 \cdot (+,-,1,1,2,2,+,-) = (3,3,1,1,2,2,4,4)$
\end{enumerate}

The action of $s_n$ is described as follows:
\begin{enumerate}
	\item If $c_n$ and $c_{n+1}$ are unequal natural numbers, with the mate of $c_n$ occurring to the left of the mate for $c_{n+1}$, then $s_n \cdot \gamma$ is obtained from $\gamma$ by operation (d).
	\item If $c_n$ and $c_{n+1}$ are opposite signs, then $s_n \cdot \gamma$ is obtained from $\gamma$ by operation (e).
	\item If neither of the above hold, then $s_n \cdot \gamma = \gamma$.
\end{enumerate}

Examples of (1), (2) above include
\begin{enumerate}
	\item $s_4 \cdot (+,1,+,1,2,-,2,-) = (+,1,+,2,1,-,2,-)$
	\item $s_4 \cdot (1,+,1,+,-,2,-,2) = (1,+,1,3,3,2,-,2)$
\end{enumerate}

Note that in the definition of $s_i \cdot \gamma$ for $i=1,\hdots,n-1$, rule (4) above is marked with an asterisk.  This rule warrants special attention, since it is the sole case in which $\ga_i$ is a non-compact imaginary root of type II.  Thus if $\gamma'$ is obtained from $\gamma$ by application of rule (4), there is a double edge connecting $\gamma$ to $\gamma'$ in the weak order graph.

\begin{definition}
For any $w \in W$, and for any skew-symmetric $(n,n)$-clan $\gamma$, define $D(w,\gamma)$ to be the number of times rule (4) must be applied in the computation of $w \cdot \gamma$.  If $w \cdot Y_{\gamma} = G/B$ and $l(w) = \text{codim}(Y_{\gamma})$, then this corresponds to the number of double edges in a path from $Y_{\gamma}$ to $G/B$ in the weak order graph, the product of whose edge labels is $w$.  (Again, this is well-defined by \cite[Lemma 5]{Brion-01}.)
\end{definition}

As with \cite[Theorem 7.5]{Wyser-11a}, when we combine Theorem \ref{thm:brion} with Corollary \ref{cor:type-c-richardson} and the case-specific combinatorics described above, we get a special case rule for Schubert constants in type $C$.  The statement of the rule is as follows:
\begin{theorem}\label{thm:type_c_structure_constants}
Suppose $(u,v)$ is a type $C$ pair of signed shuffles, with $w_0u \geq v$.  Let $u',v'$ be the images of $u,v$ in $S_{2n}$, and let $\gamma = \gamma(u',v')$ be the corresponding skew-symmetric $(n,n)$-clan avoiding the pattern $(1,2,1,2)$ (cf. Proposition \ref{prop:type-a-richardson-converse}).  Let 
\[ \gamma_0=(1,2,\hdots,n-1,n,n,n-1,\hdots,2,1) \]
be the skew-symmetric $(n,n)$-clan corresponding to the open, dense $K$-orbit on $X$.

For any $w \in W$,
\[ c_{u,v}^w = 
\begin{cases}
	2^{D(w,\gamma)} & \text{ if $l(w) = l(u) + l(v)$ and $w \cdot \gamma = \gamma_0$} \\
	0 & \text{ otherwise.}
\end{cases}
\]
\end{theorem}

\begin{example}\label{ex:example-2}
Consider the product $S_u \cdot S_v$ for $u = \overline{4}123$, $v=1\overline{4}23$.  The images of $u$ and $v$ in $S_8$ are $u' = 51236784$ and $v' = 15236748$, respectively.  One checks that this $(4,4)$-pair corresponds to the skew-symmetric $(4,4)$-clan $\gamma(u',v') = (+,-,1,2,2,1,+,-)$.  As elements of $W$, we have that $l(u) = 4$, $l(v) = 3$, and there are $44$ elements of $W$ of length $l(u) + l(v) = 7$.  Table 1 of the Appendix shows each of these elements as words in the simple reflections, the clan obtained from computing the action of each on the clan $\gamma(u,v)$, and the corresponding structure constant $c_{u,v}^w$ specified by Theorem \ref{thm:type_c_structure_constants}.

The data in Table 1, obtained using Theorem \ref{thm:type_c_structure_constants}, was seen to agree with the output of Maple code written by Alexander Yong, which computes Schubert products in all classical Lie types (\cite{Yong-Maple}).

Note that unlike the rule of \cite[Theorem 7.5]{Wyser-11a} (and the rule described in the next section for type $D$), the type $C$ rule is not multiplicity-free, as this example demonstrates.  Consider the monoidal action of $w=[3,2,1,4,3,2,1]$ on the clan $\gamma(u,v)$:
\[ (+,-,1,2,2,1,+,-) \stackrel{1}{\longrightarrow} (3,3,1,2,2,1,4,4) \stackrel{2}{\longrightarrow} (3,1,3,2,2,4,1,4) \stackrel{3}{\longrightarrow} \]
\[ (3,1,2,3,4,2,1,4) \stackrel{4}{\longrightarrow} (3,1,2,4,3,2,1,4) \stackrel{1}{\longrightarrow} (1,3,2,4,3,2,4,1) \stackrel{2}{\longrightarrow} \]
\[  (1,2,3,4,3,4,2,1) \stackrel{3^*}{\longrightarrow} (1,2,3,4,4,3,2,1). \]
Note that the last step, marked with an asterisk, is an application of rule (4) above.  Since this is the lone application of rule (4) required in the computation of the action of $w$, the corresponding structure constant is $2$.

On the other hand, the action of $[4,3,2,1,4,3,2]$ requires no applications of rule (4), so the corresponding structure constant is $1$:
\[ (+,-,1,2,2,1,+,-) \stackrel{2}{\longrightarrow} (+,1,-,2,2,+,1,-) \stackrel{3}{\longrightarrow} (+,1,2,-,+,2,1,-) \stackrel{4}{\longrightarrow} \]
\[ (+,1,2,3,3,2,1,-) \stackrel{1}{\longrightarrow} (1,+,2,3,3,2,-,1) \stackrel{2}{\longrightarrow} (1,2,+,3,3,-,2,1) \stackrel{3}{\longrightarrow} \]
\[ (1,2,3,+,-,3,2,1) \stackrel{4}{\longrightarrow} (1,2,3,4,4,3,2,1). \]
\end{example}

\section{Type D}\label{sect:type-d}
In parallel with the previous section, here we realize the even special orthogonal group $G = SO(2n,\C)$ as the isometry group of the quadratic form
\[ \left\langle x,y \right\rangle = \displaystyle\sum_{i=1}^{2n} x_i y_{2n+1-i}. \]
Thus $G$ is the set of all matrices $g$ such that $g^t J_{2n} g = J_{2n}$, with $J_{2n}$ the $2n \times 2n$ antidiagonal matrix $(\delta_{i,2n+1-j})$.

Again, the flag variety $X$ for $G$ is naturally a closed subvariety of the type $A$ flag variety $X'$, here as one of the two components of the variety of isotropic flags with respect to the form $\langle \cdot, \cdot \rangle$.  (We choose $X$ to be the component containing the ``standard flag" $\left\langle e_1,\hdots,e_{2n} \right\rangle$.)  The embedding corresponds again to the map $G/B \hookrightarrow G'/B'$, where $G'=GL(2n,\C)$, $B'$ is the Borel subgroup of $G'$ consisting of upper triangular matrices, and $B = B' \cap G$ is the Borel subgroup of $G$ consisting of upper triangular elements of $G$.  And again, in exactly the same notation as the previous section, $G$ is stable under $\theta'$, with $G^{\theta} = G \cap K' \cong GL(n,\C)$.

As in the last section, it is the case that if $Y'$ is a $K'$-orbit closure on $X'$ which coincides with a Richardson subvariety of $X'$, and if $Y' \cap X \neq \emptyset$, then $Y' \cap X$ is a single $K$-orbit closure on $X$ which coincides with a Richardson subvariety of $X$.  Applying Theorem \ref{thm:brion} once more, we obtain a positive rule for structure constants in type $D$.  The specific constants $c_{u,v}^w$ for which we get a rule once again correspond to pairs $(u,v)$ of shuffles.  When $n$ is even, this is a pair of signed shuffles as defined in the previous section, but when $n$ is odd, the definition is slightly different.  (See Definition \ref{def:type-d-shuffles}.)

We start once again by parametrizing the $K$-orbits on $X$ by a subset of the $(n,n)$-clans parametrizing $K'$-orbits on $X'$.  We do not prove the following proposition, but refer the reader to \cite{McGovern-Trapa-09,Matsuki-Oshima-90}.

\begin{prop}[\cite{McGovern-Trapa-09,Matsuki-Oshima-90}]\label{prop:type-D-orbit-param}
Let $\gamma=(c_1,\hdots,c_{2n})$ be an $(n,n)$-clan, with $Q_{\gamma}'$ the corresponding $K'$-orbit on $X'$.  Then $Q_{\gamma}' \cap X \neq \emptyset$ if and only if $\gamma$ has the following three properties:
\begin{enumerate}
	\item $\gamma$ is skew-symmetric.
	\item If $c_i = c_j$ is a pair of equal natural numbers, then $j \neq 2n+1-i$.
	\item Among $(c_1,\hdots,c_n)$, the total number of $-$ signs and pairs of equal natural numbers is even.
\end{enumerate}
Moreover, in this event, $Q_{\gamma}' \cap X$ is once again a \textit{single} $K$-orbit on $X$.  Thus $K \backslash X$ is parametrized by the set of all $(n,n)$-clans having properties (1)-(3) above.
\end{prop}

We note further (\cite{McGovern-Trapa-09,Matsuki-Oshima-90}) that when we parametrize $K \backslash X$ as described in the previous proposition, the open dense $K$-orbit corresponds to the clan
\[ \gamma_0 = (1,2,\hdots,n-3,n-2,n-1,n,n-1,n,n-3,n-2,\hdots,1,2) \]
if $n$ is even, and to the clan
\[ \gamma_0 = (1,2,\hdots,n-3,n-2,n-1,n,+,-,n-1,n,n-3,n-2,\hdots,1,2) \]
if $n$ is odd.

\begin{definition}
For the sake of brevity, we refer to an $(n,n)$-clan having properties (1)-(3) of the previous proposition as a \textbf{type D clan}.
\end{definition}

Now suppose that $Y_{\gamma}' = \overline{Q_{\gamma}'}$ is a $K'$-orbit closure on $X'$, with $\gamma$ a type $D$ clan avoiding the pattern $(1,2,1,2)$.  We know that $Y_{\gamma}'$ is a Richardson variety in $X'$.  By the previous proposition, we also know that $Y_{\gamma} := \overline{Q_{\gamma}} = \overline{Q_{\gamma}' \cap X} = Y_{\gamma}' \cap X$.  We wish to see that $Y_{\gamma}$ is a Richardson variety in $X$.

Let $W$ be the Weyl group for $G$, $W' = S_{2n}$ the Weyl group for $G'$.  In this case, $W$ consists of signed permutations of $[n]$, changing an \textit{even} number of signs.  These embed into $W'$ in precisely the same way as described in the previous section.  The only difference is that now the image of this embedding is smaller.  Indeed, elements of $W'$ which are images of $W$ under the embedding are signed elements $w'$ of $S_{2n}$ having the property that $\#\{i \leq n \mid w'(i) > n\}$ is even.  We will refer to such elements as ``type $D$ elements of $S_{2n}$".  Once again, we use notation such as $w',u'$, etc., to indicate the images of elements $w$, $u$, etc. of $W$ in $W'$. 

\begin{prop}
Suppose that $\gamma=(c_1,\hdots,c_{2n})$ is a type $D$ clan avoiding the pattern $(1,2,1,2)$.  Then $u(\gamma),v(\gamma) \in W'$ are type $D$ elements of $S_{2n}$.
\end{prop}
\begin{proof}
Since $\gamma$ is skew-symmetric, we know by Proposition \ref{prop:skew_symmetric_clans} that $u := u(\gamma)$ and $v := v(\gamma)$ are signed elements of $S_{2n}$.  Thus the only question is whether $\#\{i \leq n \mid u(i) > n\}$ and $\#\{i \leq n \mid v(i) > n\}$ are both even.

For any skew-symmetric clan avoiding the pattern $(1,2,1,2)$, and for $i \leq n$, suppose that $c_i$ is a natural number.  Then the mate for $c_i$ must be either $c_{2n+1-i}$, or $c_j$ for another $j \leq n$.  Indeed, suppose that $c_i = c_j$ for $j > n$ but $j \neq 2n+1-i$.  Then we must have  either $(c_i, c_{2n+1-j}, c_j, c_{2n+1-i})$ (if $j < 2n+1-i$) or $(c_{2n+1-j},c_i,c_{2n+1-i},c_j)$ (if $j > 2n+1-i$) in the pattern $(1,2,1,2)$.  Now, by part (2) of the definition of a type $D$ clan, the case $c_i = c_{2n+1-i}$ is not allowed.  Since $\gamma$ is a type $D$ clan, any natural number occurring among the first $n$ characters of $\gamma$ has its mate also occurring among the first $n$ characters.

Now, recall that $u$ is defined by assigning numbers less than or equal to $n$ to $+$ signs and second occurrences of natural numbers, and numbers greater than $n$ to $-$ signs and first occurrences of natural numbers.  By the observation of the previous paragraph, the number of first occurrences of natural numbers among the first $n$ characters of $\gamma$ is the same as the number of pairs of equal natural numbers among the first $n$ characters.  By part (3) of the definition of a type $D$ clan, then, we have that the total number of $-$ signs and first occurrences of natural numbers occurring among the first $n$ characters is even.  Thus $\#\{i \leq n \mid u(i) > n\}$ is even.  A similar argument, involving second occurrences of natural numbers, applies to $v$.
\end{proof}

Suppose that $Y'$ is a $K'$-orbit closure on $X'$ coinciding with a Richardson variety, such that $Y := Y' \cap X \neq \emptyset$.  As we have just seen, $Y'$ is of the form $X_{u'}^{v'}$.  We wish now to see that $Y = X_{u'}^{v'} \cap X = X_u^v$, a Richardson variety of $X$.  This is a bit more difficult to see in the type $D$ case than it was in the type $C$ case.  The complication comes from the fact that, unlike in type $C$, here the Bruhat order on $W$ is not induced by the Bruhat order on $W'$.  Indeed, the Bruhat order on $W$ is weaker than the induced order.  The precise statement is as follows:

\begin{proposition}[{\cite[\S 3.5]{Billey-Lakshmibai-00}}]\label{prop:type_d_bruhat_order}
For $1 \leq i \leq 2n$, let $i' = 2n+1-i$, and let $|i| = \text{min}(i,i')$.  For $u,v \in W$, denote by $u',v'$ their images in $W'$.  Then $u < v$ in the Bruhat order on $W$ if and only if the following two conditions hold:
\begin{enumerate}
	\item For any $d$, $1 \leq d \leq n$, if the sets $\{u'(1),\hdots,u'(d)\}$ and $\{v'(1),\hdots,v'(d)\}$ are each rewritten in ascending order, then each element of the first set is less than or equal to the corresponding element of the second set.  Equivalently, $u' < v'$ in the Bruhat order on $W'$.
	\item Suppose that $a_1 \hdots a_{2n}$ is the one-line notation for $u'$, and that $b_1 \hdots b_{2n}$ is the one-line notation for $v'$.  For any $d \in [n]$, denote by $(c_1,\hdots,c_d)$ the numbers $a_1,\hdots,a_d$ rewritten in ascending order, and by $(e_1,\hdots,e_d)$ the numbers $b_1,\hdots,b_d$ rewritten in ascending order.  Then for any pair of compatible continuous subsequences $\{|c_{i+1}|,\hdots,|c_{i+r}|\}$, $\{|e_{i+1}|,\hdots,|e_{i+r}|\}$, which are both equal, as sets (i.e. without respect to order), to $\{n+1-r,\hdots,n\}$, the numbers
	\[ \#\{j \mid i+1 \leq j \leq i+r \text{ and } c_j > n\} \]
	and
	\[ \#\{j \mid i+1 \leq j \leq i+r \text{ and } e_j > n\} \]
	are either both even or both odd.
\end{enumerate}
\end{proposition}

This complicates matters somewhat, because the type $D$ analogue of Corollary \ref{cor:type_c_schubert_vars} does not hold in general.  Indeed, using notation as in the proof of that corollary, if $C_w = D_{w'} \cap X$ is a Schubert cell of $X$, and if $D_{u'} < D_{w'}$ in the Bruhat order on Schubert cells of $X'$, it is no longer necessarily the case that $C_u < C_w$ in the Bruhat order on Schubert cells of $X$.  The upshot is that in general, the intersection of a Schubert variety $X_{w'}$ of $X'$ with $X$ is not the Schubert variety $X_w$ of $X$ --- indeed, such an intersection need only be a \textit{union} of Schubert varieties.

However, we are dealing with type $D$ elements $u',v'$ of a very specific type.  Namely, $u'$ is a shuffle of $n,n-1,\hdots,1$ and $2n,2n-1,\hdots,n+1$, while $v'$ is a shuffle of $1,\hdots,n$ and $n+1,\hdots,2n$.  The next proposition says that for such type $D$ elements of $S_{2n}$, the potential complications arising due to the weaker Bruhat order in type $D$ do not occur.

\begin{prop}\label{prop:type_d_bruhat_for_special_perms}
Suppose that $u \in W$ is such that $u' \in W'$ is a shuffle of $n,n-1,\hdots,1$ and $2n,2n-1,\hdots,n+1$, and that $v \in W$ is such that $v' \in W'$ is a shuffle of $1,\hdots,n$ and $n+1,\hdots,2n$.  For any element $w \in W$, we have that $w < u$ (resp. $w > v$) as elements of $W$ if and only if $w' < u'$ (resp. $w' > v'$) as elements of $W'$.
\end{prop}
\begin{proof}
By Proposition \ref{prop:type_d_bruhat_order}, if $w < u$, we automatically have $w' < u'$, so we need only prove the converse.  What must be verified is that condition (2) of Proposition \ref{prop:type_d_bruhat_order} automatically holds if $w' < u'$, provided that $u'$ is a shuffle of the specified type.

Let $a_1,\hdots,a_n$ be the first $n$ values of $w'$, and $b_1,\hdots,b_n$ the first $n$ values of $u'$.    Choose a $d$ with $1 \leq d \leq n$.  Let $(e_1,\hdots,e_d)$ be $b_1,\hdots,b_d$ rewritten in ascending order, and let $(c_1,\hdots,c_d)$ be $a_1,\hdots,a_d$ rewritten in ascending order.  Since we know that $b_1,\hdots,b_n$ is a shuffle of $n,\hdots,k+1$ and $2n,2n-1,\hdots,2n-k+1$ for some $k$, the sequence $(e_1,\hdots,e_d)$ must be of the form
\[ (l,l+1,\hdots,n,2n-m+1,2n-m+2,\hdots,2n) \]
for some $k+1 \leq l \leq n$ and $m \leq k$.  Taking absolute values as defined in the statement of Proposition \ref{prop:type_d_bruhat_order}, we have the sequence
\[ (l,l+1,\hdots,n,m,m-1,\hdots,1). \]

Consider the possible length $r$ continuous subsequences of this sequence which are, as sets, of the form $\{n-r+1,n-r+2,\hdots,n\}$.  One possibility is a subsequence of the form $(l',l'+1,\hdots,n)$ with $l' \geq l$.  Clearly, no elements of such a subsequence correspond to elements of the sequence $(e_1,\hdots,e_d)$ which are greater than $n$.  Since $w' < u'$, each $c_i$ in the matching subsequence of $(c_1,\hdots,c_d)$ is less than or equal to the corresponding $e_i$.  In particular, none of the $c_i$ are larger than $n$ either.  So if the matching subsequence of the $|c_i|$ also gives $\{n-r+1,\hdots,n\}$, condition (2) of Proposition \ref{prop:type_d_bruhat_order} is satisfied.

The other possibility occurs only when $d=n$.  In this case, the sequence $(e_1,\hdots,e_n)$ is $(k+1,\hdots,n,2n-k+1,\hdots,2n)$, with absolute value sequence $(k+1,k+2,\hdots,n,k,k-1,\hdots,1)$.  Here, any subsequence
\[ (k+1,k+2,\hdots,n,k,k-1,\hdots,k-j) \]
for some $j$ with $0 \leq j \leq k-1$ is of the appropriate form.  This absolute value sequence corresponds to the sequence $(e_1,\hdots,e_{n-k+j+1})$.  The elements of this sequence switch from being less than or equal to $n$ to being greater than $n$ at position $n-k+1$ --- that is, $(e_1,\hdots,e_{n-k})$ are all less than or equal to $n$, and $(e_{n-k+1},\hdots,e_{n-k+j+1})$ are all greater than $n$.  Consider the corresponding sequence $(c_1,\hdots,c_{n-k+j+1})$.  This sequence must switch from being less than or equal to $n$ to being greater than $n$ at a position at or beyond $n-k+1$ (since each $c_i$ is less than or equal to the corresponding $e_i$), which is congruent modulo $2$ to $n-k+1$ (since the number of elements of $(e_1,\hdots,e_n)$ and the number of elements of $(c_1,\hdots,c_n)$ greater than $n$ must both be even).  From these two facts, one sees that the only way that the numbers
\[ \#\{i \mid 1 \leq i \leq n-k+j+1 \text{ and } e_i > n\} \]
and
\[ \#\{i \mid 1 \leq i \leq n-k+j+1 \text{ and } c_i > n\} \]
can differ in parity is if the first is odd, and the second is zero.  If the first number is odd, then there are an odd number of terms occurring after position $n-k+j+1$, since the number of $e_i$ greater than $n$ must be even, and since all $e_i$ beyond this position are greater than $n$ (the $e_i$ being ordered).  This means that the number of $c_i$ beyond position $n-k+j+1$ which are \textit{less than} $n$ is odd, since the even number of $c_i$ which are greater than $n$ all occur beyond this position.  In particular, there is at least one $c_i$ beyond position $n-k+j+1$ which is less than $n$.  But this says that $\{|c_1|,\hdots,|c_{n-k+j+1}|\}$ cannot possibly be equal to $\{k-j,\hdots,n\}$ as a set, since all elements of it are less than $n$ (the $c_i$ being ordered).  Thus we see that condition (2) of Proposition \ref{prop:type_d_bruhat_order} cannot be violated.

We apply a similar argument for $v$.  Now let $a_1,\hdots,a_n$ be the first $n$ values of $v'$, and $b_1,\hdots,b_n$ the first $n$ values of $w'$.  Choose $d$ and define $(c_1,\hdots,c_d)$ and $(e_1,\hdots,e_d)$ as above.  We know that $a_1,\hdots,a_n$ is a shuffle of $1,\hdots,j$ and $n+1,\hdots,2n-j$, so the sequence $(c_1,\hdots,c_d)$ must be of the form
\[ (1,\hdots,l,n+1,\hdots,n+m) \]
for $1 \leq l \leq j$ and $1 \leq m \leq n-j$.  Taking absolute values, we obtain
\[ (1,\hdots,l,n,n-1,\hdots,n-m+1). \]

One possible continuous subsequence of this sequence which gives $\{n-r+1,\hdots,n\}$ as a set is $n,n-1,\hdots,n-r+1$ for $r \leq m$.  Every element of such a sequence corresponds to an element of the sequence $(c_1,\hdots,c_d)$ which is greater than $n$.  Since $w' > v'$, each $e_i$ in the matching subsequence of $(e_1,\hdots,e_d)$ is greater than or equal to the corresponding $c_i$.  In particular, all such $e_i$ are also greater than $n$.  So if the matching subsequence of the $|e_i|$ also gives $\{n-r+1,\hdots,n\}$, condition (2) of Proposition \ref{prop:type_d_bruhat_order} is satisfied.

The other possibility occurs only when $d = n$.  In this case, the absolute value sequence for $(c_1,\hdots,c_n)$ is $(1,\hdots,j,n,n-1,\hdots,j+1)$.  Here, any subsequence of the form $(k,k+1,\hdots,j,n,n-1,\hdots,j+1) $ for some $k$ with $1 \leq k \leq j$ is of the appropriate form.  This corresponds to the sequence $(c_k,\hdots,c_n)$.  This sequence contains all $c_i$ which are greater than $n$, necessarily an even number.  Since each $e_i$ is greater than or equal to the corresponding $c_i$, the only way that the numbers
\[ \#\{i \mid k \leq i \leq n \text{ and } c_i > n\} \]
and
\[ \#\{i \mid k \leq i \leq n \text{ and } e_i > n\} \]
can differ in parity is if the second number is odd.  Since an even number of $(e_1,\hdots,e_n)$ are greater than $n$, and since the $e_i$ are ordered, this can only occur if all $(e_k,\hdots,e_n)$ are greater than $n$ and there are an odd number of them.  In this case, there must also be an odd number of $e_i > n$ with $i < k$.  In particular, there is at least one such $e_i$.  Thus the absolute value set $\{|e_k|,\hdots,|e_n|\}$ cannot possibly be equal to $\{k,\hdots,n\}$, since all elements of the sequence $(e_k,\hdots,e_n)$ are at least $n+2$, and hence have absolute value at most $n-1$.  Thus again, we see that condition (2) of Proposition \ref{prop:type_d_bruhat_order} cannot be violated in this situation.  This completes the proof.
\end{proof}

\begin{corollary}\label{cor:type-d-richardson}
Suppose that $\gamma$ is a type $D$ clan avoiding the pattern $(1,2,1,2)$.  Letting $u' = u(\gamma)$, $v' = v(\gamma)$ be the corresponding type $D$ elements of $S_{2n}$, the $K$-orbit closure $Y_{\gamma} := \overline{Q_{\gamma}}$ is the Richardson variety $X_u^v$.
\end{corollary}
\begin{proof}
One argues just as in the proof of Corollary \ref{cor:type_c_schubert_vars} that for any $w \in W$, the Schubert cell $C_w$ of $X$ is precisely $D_{w'} \cap X$, with $D_{w'}$ the type $A$ Schubert cell corresponding to $w'$.  Proposition \ref{prop:type_d_bruhat_for_special_perms} allows us to argue just as in the proof of Corollary \ref{cor:type_c_schubert_vars} that $X_u = X_{u'} \cap X$, and that $X^v = X^{v'} \cap X$.  Thus
\[ Y_{\gamma} = Y_{\gamma}' \cap X = X_{u'}^{v'} \cap X = (X_{u'} \cap X) \cap (X^{v'} \cap X) = X_u \cap X^v = X_u^v. \]
\end{proof}

The previous corollary shows that we will be able to compute the Schubert product $S_{w_0u} \cdot S_v$ by once again applying Theorem \ref{thm:brion}.  Observe that $v$, as a signed permutation, is a shuffle of $1,\hdots,k$ and $-n,\hdots,-(k+1)$ for some $k$.  (Note that $n-k$ must be even.)  If $n$ is even, then $w_0$ is the signed permutation which flips all signs, so since $u$ is a shuffle of $n,n-1,\hdots,j+1$ and $-1,\hdots,-j$ for some (even) $j$, $w_0u$ is a shuffle of $1,\hdots,j$ and $-n,\hdots,-(j+1)$.  However, if $n$ is odd, then $w_0$ is the signed permutation which sends all numbers \textit{except} $n$ to their negatives, so in that case, $w_0u$ is a shuffle of $1,\hdots,j$ and $n,-(n-1),\hdots,-(j+1)$.

\begin{definition}\label{def:type-d-shuffles}
A \textbf{type D pair of signed shuffles} is a pair $(u,v)$ such that, for some $j,k$ with $j$ and $n-k$ both even, 
\begin{enumerate}
	\item $u$ is a shuffle of $1,\hdots,j$ and $-n,\hdots,-(j+1)$ if $n$ is even, or a shuffle of $1,\hdots,j$ and $n,-(n-1),\hdots,-(j+1)$ if $n$ is odd.
	\item $v$ is a shuffle of $1,\hdots,k$ and $-n,\hdots,-(k+1)$.
\end{enumerate}
\end{definition}

Note once again that \textit{any} product $S_u \cdot S_v$ with $(u,v)$ a type $D$ pair of signed shuffles can be computed using Theorem \ref{thm:brion}, since all Richardson varieties of the appropriate form arise as $K$-orbit closures:
\begin{equation}\label{eqn:type-d-schubert-prod}
	 S_u \cdot S_v = [X_{w_0u}] \cdot [X^v] = [X_{w_0u}^v] = [X_{(w_0u)'}^{v'} \cap X] = [X_{w_0'u'}^{v'} \cap X] =  [Y_{\gamma(u',v')}' \cap X] = [Y_{\gamma(u',v')}].
\end{equation}

We remark that there is an inconsistency in notation here when $n$ is odd.  Recall that the notation $\gamma(u',v')$ previously meant the clan $\gamma$ such that $w_0 u' =u(\gamma)$ and $v' = v(\gamma)$.  (Here, $w_0$ denotes the long element of $W' = S_{2n}$, not of $W$.)  In the string of equalities above, $\gamma(u',v')$ means the clan $\gamma$ such that $w_0' u' = u(\gamma)$ and $v' = v(\gamma)$, where $w_0'$ is the image in $W'$ of the long element $w_0$ of $W$.  When $n$ is even, $w_0'$ \textit{is} the long element of $W'$, but when $n$ is odd, it is not.  (Indeed, when $n$ is odd, $w_0' = w_0 s_n$, with $w_0$ the long element of $S_{2n}$, and $s_n$ the simple transposition $(n,n+1)$).  The author prefers to offer this caveat to the reader, rather than invent a separate notation for odd $n$.

We now use the facts developed in this section to give a positive formula for structure constants $c_{u,v}^w$ when $(u,v)$ is a type $D$ pair of signed shuffles.  Once again, we must understand the $M(W)$-action on $K \backslash X$ at the level of type $D$ clans.  As in the previous section, we describe this action in purely combinatorial terms, as a sequence of operations on type $D$ clans.  References are \cite{Matsuki-Oshima-90,McGovern-Trapa-09}.

Let $\frt$ be the Cartan subalgebra of $\frg = \text{Lie}(G)$ consisting of diagonal matrices 
\[ \text{diag}(a_1,\hdots,a_n,-a_n,\hdots,-a_1). \]
Let $x_1,\hdots,x_n$ be coordinates on $\frt$, with 
\[ x_i(\text{diag}(a_1,\hdots,a_n,-a_n,\hdots,-a_1)) = a_i. \]
Order the simple roots in the following way:  $\ga_i = x_i - x_{i+1}$ for $i=1,\hdots,n-1$, and $\ga_n = x_{n-1} + x_n$.  For $i=1,\hdots,n$, let $s_i$ denote $s_{\ga_i}$.  Let $\gamma=(c_1,\hdots,c_{2n})$ be any type $D$ clan.  Given a simple reflection $s_i$ for $i=1,\hdots,n-1$, consider the following possible operations on $\gamma$:
\begin{enumerate}[(a)]
	\item Interchange the characters in positions $i,i+1$ \textit{and} the characters in positions $2n-i,2n-i+1$.
	\item Replace the characters in positions $i,i+1$ \textit{and} the characters in positions $2n-i,2n-i+1$ each by a pair of matching natural numbers.
\end{enumerate}

Then the monoidal action of $M(W)$ of $s_i$ ($i=1,\hdots,n-1$) on $\gamma$ is as follows:
\begin{enumerate}
	\item If $c_i$ is a sign, $c_{i+1}$ is a number, and the mate for $c_{i+1}$ occurs to the right of $c_{i+1}$, $s_i \cdot \gamma$ is obtained from $\gamma$ by operation (a).
	\item If $c_i$ is a number, $c_{i+1}$ is a sign, and the mate for $c_i$ occurs to the left of $c_i$, $s_i \cdot \gamma$ is obtained from $\gamma$ by operation (a).
	\item If $c_i$ and $c_{i+1}$ are unequal natural numbers, with the mate for $c_i$ occurring to the left of the mate for $c_{i+1}$, \textit{and} if $(c_i,c_{i+1}) \neq (c_{2n-i},c_{2n-i+1})$, then $s_i \cdot \gamma$ is obtained from $\gamma$ by operation (a).
	\item If $c_i$ and $c_{i+1}$ are opposite signs, then $s_i \cdot \gamma$ is obtained from $\gamma$ by operation (b).
	\item If none of the above hold, them $s_i \cdot \gamma = \gamma$.
\end{enumerate}

We give examples of (1)-(4) above:
\begin{enumerate}
	\item $s_2 \cdot (+,-,1,1,2,2,+,-) = (+,1,-,1,2,+,2,-)$
	\item $s_2 \cdot (1,1,-,+,-,+,2,2) = (1,-,1,+,-,2,+,2)$
	\item $s_1 \cdot (1,2,1,2,3,4,3,4) = (2,1,1,2,3,4,4,3)$
	\item $s_1 \cdot (+,-,1,1,2,2,+,-) = (3,3,1,1,2,2,4,4)$
\end{enumerate}

The action of $s_n$ is a bit different.  The most concise way to define it is as follows:  Given a type $D$ clan $\gamma$, let $\text{Flip}(\gamma)$ denote the clan obtained from $\gamma$ by interchanging the characters in positions $n,n+1$.  Then
\[ s_n \cdot \gamma = \text{Flip}(s_{n-1} \cdot \text{Flip}(\gamma)). \]

When $n = 3$, we have the following examples:
\begin{enumerate}
	\item $s_3 \cdot (+,+,+,-,-,-) = (+,1,2,1,2,+)$.  We apply Flip to obtain $(+,+,-,+,-,-)$, act by $s_2$ on the result to obtain $(+,1,1,2,2,+)$, and finally apply Flip once more to obtain $(+,1,2,1,2,+)$.
	\item $s_3 \cdot (1,-,1,2,+,2) = (1,2,+,-,1,2)$.  We apply Flip to obtain $(1,-,2,1,+,2)$, apply $s_2$ to obtain $(1,2,-,+,1,2)$, and apply Flip again to obtain $(1,2,+,-,1,2)$.
	\item $s_3 \cdot (-,1,1,2,2,+) = (-,1,1,2,2,+)$.  We Flip to obtain $(-,1,2,1,2,+)$, apply $s_2$ to the result (which does nothing), and Flip again, which returns us to the clan we started with.
\end{enumerate}

Unlike the rule of Theorem \ref{thm:type_c_structure_constants}, our rule in type $D$ is multiplicity-free, due to the following fact:
\begin{proposition}
In the weak order graph for $K \backslash X$, all edges are single.
\end{proposition}
\begin{proof}
Similarly to the type $C$ case described in the previous section, $\ga_i$ ($i=1,\hdots,n-1$) is non-compact imaginary for $\gamma$ if and only if $c_i$ and $c_{i+1}$ are opposite signs.  The cross-action of $s_i$ in this case is to interchange the opposite signs in positions $i,i+1$ as well as those in positions $2n-i,2n-i+1$.  Thus $s_i \times \gamma \neq \gamma$, so $s_i$ is type I.  The simple root $\ga_n$ is non-compact imaginary for $\gamma$ if and only if $(c_{n-1},c_n,c_{n+1},c_{n+2}) = (+,+,-,-)$ or $(-,-,+,+)$.  Here, the cross-action of $s_n$ is by the permutation action of $s_n' = (n-1,n+1)(n,n+2) \in S_{2n}$, thus it interchanges these two patterns.  In particular, in each case we have $s_n \times \gamma \neq \gamma$, so that $s_n$ is of type I.  Thus all non-compact imaginary roots are of type I, so that all edges in the weak order graph are single.

Alternatively, the claim here follows from \cite[Corollary 2]{Brion-01}, and indeed this case is mentioned explicitly in the discussion immediately following that corollary.
\end{proof}

The combinatorics of the monoidal action of $W$ described above, together with Corollary \ref{cor:type-d-richardson}, Theorem \ref{thm:brion}, and the previous proposition then give us the following multiplicity-free special case rule for Schubert constants in type $D$:
\begin{theorem}\label{thm:type-d-structure-consts}
Suppose $(u,v)$ is a type $D$ pair of signed shuffles, with $w_0u \geq v$.  Let $u',v'$ be the images of $u,v$ in $S_{2n}$, and let $\gamma = \gamma(u',v')$ be the corresponding type $D$ clan avoiding the pattern $(1,2,1,2)$.  (Recall our remarks on this notation immediately following displayed equation (\ref{eqn:type-d-schubert-prod}) in the event that $n$ is odd.)  Let $\gamma_0$ be the type $D$ clan corresponding to the open, dense $K$-orbit on $X$, as described immediately following the statement of Proposition \ref{prop:type-D-orbit-param}.

Then for any $w \in W$,
\[ c_{u,v}^w = 
\begin{cases}
	1 & \text{ if $l(w) = l(u) + l(v)$ and $w \cdot \gamma = \gamma_0$} \\
	0 & \text{ otherwise.}
\end{cases}
\]
\end{theorem}

\begin{example}\label{ex:example-3}
We first give an example with $n$ even.  Take $n = 4$, and consider the product $S_u \cdot S_v = S_{\overline{4} 1 2 \overline{3}} \cdot S_{1 2 \overline{4} \overline{3}}$.  This corresponds to $u' = 51263784$, $v' = 12563478$.  One computes that $\gamma(u',v')$ is the type $D$ clan $(+,1,-,1,2,+,2,-)$.  As elements of $W$, $l(u) = 3$ and $l(v) = 1$.  There are $23$ elements of $W$ of length $l(u) + l(v) = 4$.  Table 2 of the Appendix shows each of these elements as words in the simple reflections, the clan obtained from computing the action of each on the clan $\gamma(u',v')$, and the corresponding structure constant $c_{u,v}^w$ according to Theorem \ref{thm:type-d-structure-consts}.

The data in Table 2, obtained using Theorem \ref{thm:type-d-structure-consts}, was seen to agree with the output of \cite{Yong-Maple}. 
\end{example}

\begin{example}\label{ex:example-4}
For an example with $n$ odd, take $n=3$ and consider the product $S_u \cdot S_v = S_{132} \cdot S_{\overline{3}1\overline{2}}$.  This corresponds to $u' = 132546$, $v'=415263$.  Here, $\gamma(u',v')$ means the clan $\gamma$ such that $653421 \cdot u' = 635241 = u(\gamma)$ and $v' = v(\gamma)$.  One checks that this clan is $(-,+,-,+,-,+)$.  As elements of $W$, $l(u) = 1$ and $l(v) = 2$.  There are $6$ elements of $W$ of length $l(u) + l(v) = 3$.  Table 3 of the Appendix gives the details of the computation according to Theorem \ref{thm:type-d-structure-consts}.

The data in Table 3 agrees with the output of \cite{Yong-Maple}.
\end{example}

\section{A Final Question}
As we have seen, both here and in \cite{Wyser-11a}, Theorem \ref{thm:brion} applies very generally to the class of any spherical subgroup orbit closure in any flag variety.  In total, we have now seen three examples where closures of orbits of certain symmetric subgroups coincide with Richardson varieties, and have used this to obtain some limited information on Schubert calculus.  The author feels that it is natural to wonder whether there are other examples, and leaves the reader with this question.

\begin{question}
Are there other examples of spherical subgroups of the classical groups, the closures of whose orbits on the flag variety coincide with Richardson varieties?  If so, are combinatorial parametrizations of those orbits understood, and is the $M(W)$-action on the orbits understood on the level of that combinatorial parametrization?
\end{question}

\newpage

\appendix
\section{Tables for Examples}
\begin{table}[h]
	\caption{Example \ref{ex:example-2}:  Computing the $C_4$ Schubert product $S_{\overline{4}123} \cdot S_{1 \overline{4}23}$}
	\scalebox{0.92}{
	\begin{tabular}{|c|c|c|}
		\hline
		Length $7$ Element $w$ & $w \cdot (+,-,1,2,2,1,+,-)$ & $c_{u,v}^w$ \\ \hline
		$[1,2,1,4,3,2,1]$ & $(1,2,3,4,3,4,2,1)$  & $0$ \\ \hline
		$[3,2,1,4,3,2,1]$ & $(1,2,3,4,4,3,2,1)$ & $2$ \\ \hline
		$[1,3,2,4,3,2,1]$ & $(1,2,3,4,4,2,3,1)$ & $0$ \\ \hline
		$[2,3,2,4,3,2,1]$ & $(1,2,3,4,4,3,1,2)$ & $0$ \\ \hline
		$[2,1,3,4,3,2,1]$ & $(1,2,3,4,4,3,2,1)$ & $2$ \\ \hline
		$[1,2,3,4,3,2,1]$ & $(1,2,3,4,4,3,2,1)$ & $2$ \\ \hline
		$[1,3,2,1,4,3,2]$ & $(1,2,3,+,-,3,2,1)$ & $0$ \\ \hline
		$[2,3,2,1,4,3,2]$ & $(1,2,3,+,-,3,2,1)$ & $0$ \\ \hline
		$[4,3,2,1,4,3,2]$ & $(1,2,3,4,4,3,2,1)$ & $1$ \\ \hline
		$[2,1,3,2,4,3,2]$ & $(1,2,+,3,3,-,2,1)$ & $0$ \\ \hline
		$[1,2,3,2,4,3,2]$ & $(1,+,2,3,3,2,-,1)$ & $0$\\ \hline
		$[1,2,1,3,4,3,2]$ & $(1,2,+,3,3,-,2,1)$ & $0$ \\ \hline
		$[2,1,3,2,1,4,3]$ & $(1,2,3,3,4,4,2,1)$ & $0$ \\ \hline
		$[1,2,3,2,1,4,3]$ & $(1,2,3,2,4,3,4,1)$ & $0$ \\ \hline
		$[1,4,3,2,1,4,3]$ & $(1,2,3,4,2,3,4,1)$ & $0$ \\ \hline
		$[2,4,3,2,1,4,3]$ & $(1,2,3,4,1,3,2,4)$ & $0$\\ \hline
		$[3,4,3,2,1,4,3]$ & $(1,2,3,4,4,1,2,3)$ & $0$ \\ \hline
		$[1,2,1,3,2,4,3]$ & $(1,2,+,-,+,-,2,1)$  & $0$ \\ \hline
		$[2,1,4,3,2,4,3]$ & $(1,2,+,3,3,-,2,1)$ & $0$ \\ \hline
		$[1,2,4,3,2,4,3]$ & $(1,+,2,3,3,2,-,1)$ & $0$ \\ \hline
		$[3,2,4,3,2,4,3]$ & $(+,1,2,3,3,2,1,-)$ & $0$ \\ \hline
		$[1,3,4,3,2,4,3]$ & $(1,+,2,3,3,2,-,1)$ & $0$ \\ \hline
		$[2,3,4,3,2,4,3]$ & $(+,1,2,3,3,2,1,-)$ & $0$ \\ \hline
		$[1,2,1,3,2,1,4]$ & $(1,2,3,3,4,4,2,1)$ & $0$ \\ \hline
		$[2,1,4,3,2,1,4]$ & $(1,2,3,4,3,4,2,1)$ & $0$ \\ \hline
		$[1,2,4,3,2,1,4]$ & $(1,2,3,4,2,3,4,1)$ & $0$ \\ \hline
		$[3,2,4,3,2,1,4]$ & $(1,2,3,4,4,1,2,3)$ & $0$ \\ \hline
		$[1,3,4,3,2,1,4]$ & $(1,2,3,4,4,2,3,1)$ & $0$\\ \hline
		$[2,3,4,3,2,1,4]$ & $(1,2,3,4,4,3,1,2)$ & $0$ \\ \hline
		$[1,2,1,4,3,2,4]$ & $(1,2,+,3,3,-,2,1)$ & $0$ \\ \hline
		$[3,2,1,4,3,2,4]$ & $(1,2,3,+,-,3,2,1)$ & $0$ \\ \hline
		$[1,3,2,4,3,2,4]$ & $(1,+,2,3,3,2,-,1)$ & $0$ \\ \hline
		$[2,3,2,4,3,2,4]$ & $(+,1,2,3,3,2,1,-)$ & $0$  \\ \hline
		$[2,1,3,4,3,2,4]$ & $(1,2,+,3,3,-,2,1)$ & $0$ \\ \hline
		$[1,2,3,4,3,2,4]$ & $(1,+,2,3,3,2,-,1)$ & $0$ \\ \hline
		$[1,3,2,1,4,3,4]$ & $(1,2,3,2,4,3,4,1)$ & $0$\\ \hline
		$[2,3,2,1,4,3,4]$ & $(1,2,3,1,4,3,2,4)$ & $0$ \\ \hline
		$[4,3,2,1,4,3,4]$ & $(1,2,3,4,1,3,2,4)$ & $0$ \\ \hline
		$[2,1,3,2,4,3,4]$ & $(1,2,+,-,+,-,2,1)$ & $0$ \\ \hline
		$[1,2,3,2,4,3,4]$ & $(1,+,2,-,+,2,-,1)$ & $0$  \\ \hline
		$[1,4,3,2,4,3,4]$ & $(1,+,2,3,3,2,-,1)$ & $0$ \\ \hline
		$[2,4,3,2,4,3,4]$ & $(+,1,2,3,3,2,1,-)$ & $0$ \\ \hline
		$[3,4,3,2,4,3,4]$ & $(+,1,2,3,3,2,1,-)$ & $0$\\ \hline
		$[1,2,1,3,4,3,4]$ & $(1,2,2,3,3,4,4,1)$ & $0$ \\ \hline
	\end{tabular}
	}
\end{table}

\newpage

\begin{table}[h]
	\caption{Example \ref{ex:example-3}:  Computing the $D_4$ Schubert product $S_{\overline{4} 1 2 \overline{3}} \cdot S_{1 2 \overline{4} \overline{3}}$}
	\begin{tabular}{|c|c|c|}
		\hline
		Length $4$ Element $w$ & $w \cdot (+,1,-,1,2,+,2,-)$ & $c_{u,v}^w$ \\ \hline
		$[1,3,2,1]$ & $(1,2,2,1,3,4,4,3)$  & $0$ \\ \hline
		$[2,3,2,1]$ & $(1,2,2,1,3,4,4,3)$ & $0$ \\ \hline
		$[1,4,2,1]$ & $(1,2,3,4,2,1,4,3)$ & $0$ \\ \hline
		$[2,4,2,1]$ & $(1,2,3,4,3,4,1,2)$ & $1$ \\ \hline
		$[3,4,2,1]$ & $(1,2,3,4,2,1,4,3)$ & $0$ \\ \hline
		$[2,1,3,2]$ & $(1,2,2,1,3,4,4,3)$ & $0$ \\ \hline
		$[1,2,3,2]$ & $(1,+,-,1,2,+,-,2)$ & $0$ \\ \hline
		$[2,1,4,2]$ & $(1,2,+,+,-,-,1,2)$ & $0$ \\ \hline
		$[1,2,4,2]$ & $(1,+,2,+,-,1,-,2)$ & $0$ \\ \hline
		$[3,2,4,2]$ & $(+,1,2,+,-,1,2,-)$ & $0$ \\ \hline
		$[1,3,4,2]$ & $(1,+,2,+,-,1,-,2)$ & $0$\\ \hline
		$[2,3,4,2]$ & $(+,1,2,+,-,1,2,-)$ & $0$ \\ \hline
		$[1,2,1,3]$ & $(1,2,2,1,3,4,4,3)$ & $0$ \\ \hline
		$[4,2,1,3]$ & $(1,2,3,4,2,1,4,3)$ & $0$ \\ \hline
		$[1,4,2,3]$ & $(1,+,2,+,-,1,-,2)$ & $0$ \\ \hline
		$[2,4,2,3]$ & $(+,1,2,+,-,1,2,-)$ & $0$\\ \hline
		$[3,4,2,3]$ & $(+,1,2,+,-,1,2,-)$ & $0$ \\ \hline
		$[1,2,1,4]$ & $(1,2,+,+,-,-,1,2)$  & $0$ \\ \hline
		$[3,2,1,4]$ & $(1,2,+,+,-,-,1,2)$ & $0$ \\ \hline
		$[1,3,2,4]$ & $(1,+,2,+,-,1,-,2)$ & $0$ \\ \hline
		$[2,3,2,4]$ & $(+,1,2,+,-,1,2,-)$ & $0$ \\ \hline
		$[2,1,3,4]$ & $(1,2,+,+,-,-,1,2)$ & $0$ \\ \hline
		$[1,2,3,4]$ & $(1,+,2,+,-,1,-,2)$ & $0$ \\ \hline
	\end{tabular}
\end{table}

\begin{table}[h]
	\caption{Example \ref{ex:example-4}:  Computing the $D_3$ Schubert product $S_{132} \cdot S_{\overline{3}1\overline{2}}$}
	\begin{tabular}{|c|c|c|}
		\hline
		Length $3$ Element $w$ & $w \cdot (-,+,-,+,-,+)$ & $c_{u,v}^w$ \\ \hline
		$[1,2,1]$ & $(1,-,1,2,+,2)$  & $0$ \\ \hline
		$[1,3,1]$ & $(1,+,2,1,-,2)$ & $0$ \\ \hline
		$[2,3,1]$ & $(1,2,+,-,1,2)$ & $1$ \\ \hline
		$[3,1,2]$ & $(1,2,+,-,1,2)$ & $1$ \\ \hline
		$[2,1,3]$ & $(1,-,1,2,+,2)$ & $0$ \\ \hline
		$[1,2,3]$ & $(1,-,1,2,+,2)$ & $0$ \\ \hline
	\end{tabular}
\end{table}

\bibliographystyle{alpha}
\bibliography{sourceDatabase}

\end{document}